\documentclass[bjps]{imsart}

\usepackage{amsfonts,amsmath,latexsym,amssymb,mathrsfs,amsthm}

\startlocaldefs

\def\numberlikeadb{\global\def\theequation{\thesection.\arabic{equation}}}
\numberlikeadb
\newtheorem{theorem}{Theorem}[section]

\newtheorem{corollary}[theorem]{Corollary}

\newtheorem{proposition}[theorem]{Proposition}
\newtheorem{remark}[theorem]{Remark}

\newtheorem{con}[theorem]{Conjecture}

\numberwithin{equation}{section}

\newcommand{\eq}{\begin{equation}}
\newcommand{\qe}{\end{equation}}

\def\N{{\rm I\kern-0.16em N}}
\def\R{{\rm I\kern-0.16em R}}
\def\E{{\rm I\kern-0.16em E}}
\def\P{{\rm I\kern-0.16em P}}
\def\F{{\rm I\kern-0.16em F}}
\def\B{{\rm I\kern-0.16em B}}
\def\Z{{\rm I\kern-0.46em Z}}
\def\C{{\rm I\kern-0.46em C}}
\def\G{{\rm I\kern-0.50em G}}

\usepackage{color}

\endlocaldefs

\begin{document}

\begin{frontmatter}

\title{Some new Stein operators for product distributions}

\runtitle{Some new Stein operators for product distributions}

\begin{aug}

 \author{\fnms{Robert E.} \snm{Gaunt}\thanksref{a}\corref{Robert E. Gaunt}\ead[label=e1]{robert.gaunt@manchester.ac.uk}\ead[label=e2,url]{www.foo.com}}
 
 \,
 
\author{\fnms{Guillaume} \snm{Mijoule}\thanksref{b}\ead[label=e2]{A.Anastasiou@lse.ac.uk}}

\and

\author{\fnms{Yvik} \snm{Swan}\thanksref{c}\ead[label=e2]{yvswan@ulb.be}}

\affiliation[a]{The University of Manchester}
\affiliation[b]{INRIA Paris}
\affiliation[c]{Universit\'{e} libre de Bruxelles}


\runauthor{R. E. Gaunt, G. Mijoule and Y. Swan}

\end{aug}

\begin{abstract} We provide a general result for finding Stein
  operators for the product of two independent random variables whose
  Stein operators satisfy a certain assumption, extending a recent
  result of Gaunt, Mijoule and Swan \cite{gms18}.  This framework applies to non-centered
  normal and non-centered gamma random variables, as well as a general
  sub-family of the variance-gamma distributions.  Curiously, there is
  an increase in complexity in the Stein operators for products of
  independent normals as one moves, for example, from centered to
  non-centered normals.  As applications, we give a simple derivation
  of the characteristic function of the product of independent
  normals, and provide insight into why the probability density
  function of this distribution is much more complicated in the
  non-centered case than the centered case.
\end{abstract}



\end{frontmatter}

\section{Introduction}

In 1972, Charles Stein \cite{stein} introduced a powerful technique
for deriving explicit bounds in normal approximations.  Shortly after,
in 1975, Louis Chen \cite{chen 0} adapted the method to the Poisson
distribution, and since then Stein's method has been extended to a
wide variety of distributional approximations.  For a given target
distribution $p$, the first step in the general procedure is to find a
suitable operator $A$ acting on a class of functions $\mathcal{F}$
such that $\E [Af(X)] =0$ for all $f\in\mathcal{F}$, where the random
variable $X$ has distribution $p$.  The operator $A$ is called the
\emph{Stein operator}, and for continuous distributions is typically a
differential operator; for the $N(\mu,\sigma^2)$ distribution, the
classical operator is $Af(x)=\sigma^2f'(x)-(x-\mu)f(x)$.  This leads
to the \emph{Stein equation}
\begin{equation}\label{steineqn}Af_h(x)=h(x)-\E h(X),
\end{equation}
where $h$ is a real-valued function.  If $A$ is well chosen then, for
a given $h$, the Stein equation (\ref{steineqn}) can be solved for
$f_h$, and the problem of
estimating the proximity of the distribution of a random variable $W$
of interest to the distribution of the target random variable $X$, as
measured by $|\E h(W)-\E h(X)|$, reduces to one of
bounding $|\E [Af_h(W)]|$.  For a detailed account of
the method we refer the reader to the
monograph Stein \cite{stein2}.

In addition to the normal and Poisson distributions, Stein's method
has been adapted to many classical distributions, such as the
exponential (Chatterjee, Fulman and R\"ollin \cite{chatt}), gamma (Luk
\cite{luk}) and Laplace (Pike and Ren \cite{pike}), as well as quite
general families of distributions, such as the Pearson family
(Schoutens \cite{schoutens}), variance-gamma distributions (Gaunt
\cite{gaunt vg}) and a wide class of distributions satisfying a
certain diffusive assumption (D\"obler \cite{dobler beta}, Kusuoka and
Tudor \cite{kusuotud}); for an overview see Ley, Reinert and Swan
\cite{ley}. As such, over the years, a number of techniques have been
developed for finding Stein operators for a variety of distributions.
These include the density method (Stein \cite{stein2}, Ley, Reinert
and Swan \cite{ley}, Ley and Swan \cite{LS16}, Mijoule, Reinert and
Swan \cite{MRS18}), the generator method (Barbour \cite{barbour2},
G\"otze \cite{gotze}), the differential equation duality approach
(Gaunt \cite{gaunt gh}, Ley, Reinert and Swan \cite{ley}), and
probability generating function and characteristic function based
approaches of Upadhye, \v{C}ekanavi\v{c}ius and Vellaisamy
\cite{cvp17} and Arras et al$.$ \cite{aaps18}. The corpus of
literature concerning Stein operators and their applications is now
vast, and it continues growing at a steady pace. Stein operators
provide handles on target distributions which are in some sense just
as important and natural characteristics of a probability distribution
as its moments, its moment generating function, its p.d.f$.$, c.d.f$.$
or even its characteristic function. Finding tractable Stein operators
is thus, naturally, an important question.

In this paper we pursue the work begun in Gaunt \cite{gaunt pn} and Gaunt \cite{gaunt
    ngb} concerning the following question : ``given two independent
  random variables $X$ and $Y$ with Stein operators $A_X$ and $A_Y$,
  can one find a Stein operator for $Z = XY$?''  More specifically,
  the present paper is a complement (sequel) to our paper Gaunt, Mijoule and Swan \cite{gms18}
  where we developed an algebraic technique for finding Stein
  operators for products of independent random variables with
  \emph{polynomial Stein operators} satisfying a technical condition.
  Let $M(f)=(x\mapsto xf(x))$, $D(f)=(x\mapsto f'(x))$ and $I$ be the
  identity operator. We say that the absolutely continuous 
  variates $X$ and $Y$ have \emph{polynomial} Stein operators if they
  allow Stein operators of the form $A = \sum_{i,j}a_{ij}M^iD^j$ for
  $a_{ij}$ some real numbers. The highest value of $j$ such that
  $a_{ij}\neq 0$ is called the order of the operator. In Gaunt, Mijoule and Swan \cite{gms18}
  we provided a method for deriving operators under the technical
  assumption that $\# \left\{ j-i\, | \, a_{ij}\neq0 \right\}\le 2$
(see Assumption 3 and Lemma 2.6 of Gaunt, Mijoule and Swan \cite{gms18} for more details on
this condition).  For such random variables, Proposition 2.12 of Gaunt, Mijoule and Swan
\cite{gms18} gives a polynomial Stein operator for the product $XY$.
A number of classical random variables have Stein operators which
satisfy this assumption, such as the $N(0,\sigma^2)$ distribution with Stein operator $\sigma^2 D-M$, with others including the gamma,
beta, and even some more exotic distributions such
as the zero-mean symmetric variance-gamma distribution and PRR
distribution of P\"ekoz, R\"ollin and Ross \cite{pekoz}.  However, some very natural
  densities do \emph{not} satisfy the assumption. In fact, even the
non-centered normal distribution does not satisfy this assumption, as
its Stein operator $\sigma^2D+\mu I-M$ instead satisfies
$\# \left\{ j-i\, | \, a_{ij}\neq0 \right\}= 3$.  In Proposition
\ref{propsec2}, we shall address the natural problem of
extending the result of Gaunt, Mijoule and Swan \cite{gms18} to treat the product of two
independent random variables satisfying this new assumption. Here
we have only added one level of complexity in the operator;
nevertheless, as we will see later on, it is sufficient to include the
classical cases of non-centered normal and non-centered gamma, and a
more general sub-family of the variance-gamma distributions.  Also, as
noted in Remark \ref{remarksec2}, the proof technique is novel and
seems to be a useful addition to the toolkit for finding Stein
operators.

The Stein operators for the products of independent normal random
variables are particularly theoretically interesting, and we devote
Section \ref{sec3} to exploring some of their properties.   For the
case of two independent centered normals a second order Stein operator
was obtained by Gaunt \cite{gaunt pn}, whereas, rather curiously, we find a
third order operator for the product of two i.i.d$.$ normals, and a
fourth order operator for the product of two independent general
normals; see Table \ref{table:nonlin}. 
It is an important and natural question to ask whether our operators
have minimal order amongst all Stein operators with polynomial
coefficients.  We believe this is the case but are unable to prove it.
However, in Section \ref{sec3.1}, we are able to provide a brute force
approach for verifying this assertion for polynomial coefficients up
to a particular order.  This brute force approach is very general and in principle can be applied to any polynomial Stein operators.  In Section \ref{sec3.2}, we prove that our
Stein operators for products of independent normals characterise the
distribution.  We do this by appealing to a more general result, Proposition \ref{conmom}, which treats distributions that are determined by their moments.

\begin{table}[ht]
\caption{Stein operators for products of normal random variables. } 
\centering
\begin{tabular}{c l l}
\hline 
Product $P$ & Stein operator $A_Pf(x)$ (here we set $\sigma:=\sigma_X\sigma_Y$) \\ [0.5ex]
\hline
$N(0,\sigma_X^2)\times N(0,\sigma_Y^2)$ & $\sigma^2(xf''(x)+xf(x))-xf(x)$  \\ [1.5ex]
$N(\mu,\sigma_X^2)\times N(\mu,\sigma_Y^2)$ & $\sigma^3xf^{(3)}(x) + \sigma^2(\sigma-x) f''(x) - \sigma( x+(\sigma+\mu^2))f'(x)$\\
&$ + (x - \mu^2)f(x)$ \\ [1.5ex]
$N(\mu_X,\sigma_X^2)\times N(\mu_Y,\sigma_Y^2)$ & $\sigma_X^4\sigma_Y^4 xf^{(4)}(x)+\sigma_X^4\sigma_Y^4 f^{(3)}(x)-\sigma_X^2\sigma_Y^2(2x+\mu_X\mu_Y)f''(x)$ \\
& $-(\sigma_X^2\sigma_Y^2+\mu_X^2\sigma_Y^2+\mu_Y^2\sigma_X^2)f'(x)+(x-\mu_X\mu_Y)f(x)$  \\ [1ex]

\hline
\end{tabular}
\label{table:nonlin}
\end{table}

For the Stein operator of Gaunt \cite{gaunt pn} for the product of two
independent standard normal random variables, it was possible to solve
the corresponding Stein equation and bound the derivatives of the
  solution.  As a result, Gaunt \cite{gaunt pn} was able to derive explicit
bounds for product normal approximations.  However, it seems to be
beyond the scope of existing techniques in the Stein's method
literature to solve and then bound the derivatives of the solution to
our more complicated third and fourth order Stein equations for
products of non-centered normals.  It should be noted, though, that
there is still great utility to Stein equations even when it is not
possible to obtain bounds for the solution.  For example, as has been
demonstrated in several papers such as
Nourdin, Peccati and Swan \cite{nourdinpecswan2014}, Arras et al$.$ \cite{aaps16} and Arras et al$.$ \cite{aaps18},
Stein operators can be used for comparison of probability
distributions directly without solving Stein equations.  We also
stress that Stein operators are also of use in applications beyond
proving approximation theorems; for example, in obtaining
distributional properties (Gaunt \cite{gaunt pn}, Gaunt \cite{gaunt ngb}, Gaunt, Mijoule and Swan \cite{gms18}).  Indeed, in Section
\ref{sec3.3}, we use our Stein operators to obtain a simple derivation
of the characteristic function of two independent normals, and also
provide valuable insight into why there is a dramatic increase in
complexity in the probability density function from the centered to
non-centered case.

\section{New Stein operators for product distributions}
\label{sec2}

\subsection{A general result}
\label{sec:particular-case-two}

Throughout this paper, we shall make the following assumptions, which were also made in Gaunt, Mijoule and Swan \cite{gms18}; we refer the reader to that paper for some remarks on these assumptions.

\vspace{3mm}

\noindent \textbf{Assumption } 1). $X$ admits a smooth density $p$ with respect to the
Lebesgue measure on $\R$; this density is defined and non-vanishing on
some (possibly unbounded) interval $J \subseteq \R$. 2).  $X$   admits an
operator $A$ acting on $\mathcal F$ which contains the set of smooth
functions with compact support $\mathcal C_0^\infty(\R)$.  

\vspace{3mm}

Let $P$ be a real polynomial. Then it is easily proved (by checking it when $P$ is a monomial, then by linearity) that
\eq
\label{eq:pmd}
\begin{aligned}
P(MD) M = MP(MD+I),\quad \text{and} \quad
DP(MD) = P(MD+I) D
\end{aligned}
\qe (recall the notations $M(f)=(x\mapsto xf(x))$,
$D(f)=(x\mapsto f'(x))$ and $I$ the identity operator from the
introduction).  Now, for $a\in\R\setminus\{0\}$, let
$\tau_a (f) = \left( x \mapsto f(ax) \right)$.  Simple computations
show that (see Gaunt, Mijoule and Swan \cite[Lemma 2.5]{gms18}) $\tau_a M = a M \tau_a$ and $D\tau_a  = a \tau_a D.$
This implies that for any real polynomial $P$,
$$\tau_a P(MD) = P(MD) \tau_a.$$

\begin{proposition}\label{propsec2}
Let $X$ and $Y$ be i.i.d$.$ with common Stein operator of the form 
$$A = M - Q(MD) - P(MD) D$$
for $P, Q$ two real polynomials.
Then, a Stein operator for $Z = XY$ is
\eq
\label{eq:prop3140}
A_Z =  R_1(MD)D^2  + R_2(MD) D  + R_3(MD) + M R_4(MD) ,
\qe
where, for $U=MD$,
\begin{align*}
 R_1(U)&= (P(U))^2P(U+I)(U+I)Q(U+2I),\\
 R_2(U) &= -(P(U))^2Q(U)(U+I) -Q(U+I)P(U)Q^2(U),\\
 R_3(U)&= -UQ(U)P(U-I)-Q(U-I)(Q(U))^2,\\
 R_4(U) &=Q(U-I).
\end{align*} 
\end{proposition}
\begin{proof}
Let $Z = XY$ and $f \in \mathcal F$. Denote $U = MD$. We have
\begin{align}
\E[X f(Z)] &= \E[M \tau_Yf(X)]\nonumber\\
&= \E [ Q(U) \tau_Y f(X) + P(U)D \tau_Yf(X)]\nonumber\\
&= \E [\tau_Y Q(U) f(X) + Y \tau_Y P(U) D f(X)]\nonumber\\
\E[X f(Z)]& = \E [Q(U) f(Z) + Y P(U) D f(Z)]. \label{eq:xfofz}
\end{align}
Similarly,
\eq
\E[Yf(Z)] = \E [Q(U) f(Z) + X P(U) D f(Z)]. \label{eq:yfofz}
\qe
Replace $f$ with $P(U)D f$ in \eqref{eq:yfofz} and add up to \eqref{eq:xfofz} to get
$$
\E[X(I- P(U)DP(U)D)f(Z)] = \E[(Q(U) + Q(U)P(U)D ) f (Z)],
$$
which is also, using \eqref{eq:pmd},
\eq
\label{eq:fz1}
\E[X(I- P(U+I)P(U)D^2)f(Z)] = \E[(Q(U) + Q(U)P(U)D) f (Z)].
\qe
Now using \eqref{eq:yfofz} and conditioning, we can compute
\begin{align}
\E[Z f(Z)] &= \E[ X \E[Y f(Z) \, | \, X]]= \E\left[ X Q(U) f (Z) +  X^2P(U) D f(Z) \right].\label{eq:zfz}
\end{align}
We also have
\eq
\begin{aligned}
\E[X^2 f(Z)] &= \E [ M^2 \tau_Yf(X)] \nonumber\\
&= \E[ Q(U) M \tau_Y f(X) + P(U) D M\tau_Yf(X)]\nonumber\\
&= \E [M \tau_Y Q(U+I) f(X) + \tau_YP(U)(U+I)f(X)]\nonumber\\
\E[X^2f(Z)]&= \E[X Q(U+I)f(Z) + P(U)(U+I)f(Z)]. \label{eq:x2f}
\end{aligned}
\qe
Thus we obtain by \eqref{eq:zfz}
\eq
\label{eq:fz2}
\E[(M-  P(U)P(U)(U+I)D)f(Z)] = \E[X\left( Q(U)+ Q(U+I)P(U)D \right)f(Z)].
\qe
Apply \eqref{eq:fz2} to $P(U)D f$ and add up to \eqref{eq:fz1} applied to $Q(U-I) f$ to obtain
\begin{align*}
&\E[X(Q(U)P(U)D+Q(U-I))f(Z)] \\
 &\quad=\E[(UP(U-I)- (P(U))^2P(U+I)(U+I) D^2\\
 &\quad\quad+(Q(U)+Q(U)P(U)D)Q(U-I) )f(Z)].
\end{align*}
Apply the preceding equation to $Q(U)f$ and subtract to \eqref{eq:fz2} applied to $Q(U-I)f$ to get the result.
\end{proof}

The case that $P$ and $Q$ are polynomials of degree one is important, as it is applicable to non-centered normal and non-centered gamma random variables, as well as a general sub-family of the variance-gamma distributions.  To this end, let us define the operator $T_r:=MD+rI$.  We note that the limit of $T_r$ as $r\rightarrow\infty$ is ill-defined, but we do have $\lim_{r\rightarrow\infty}r^{-1}T_r=I$ (see Gaunt, Mijoule and Swan \cite[Remark 2.3]{gms18}).

\begin{corollary}\label{prop:314}Let $\alpha,\beta \in \R$  and $a,b \in \R \cup \{\infty\}$ (if either $a$ or $b$ are set to $+\infty$, then we proceed as described above). Let $X,Y$ be i.i.d$.$  with common Stein operator
$$A = M - \alpha T_a - \beta T_b D.$$
Then, a Stein operator for $Z = XY$ is
\eq
\label{eq:prop314}
A_Z = (M - \alpha^2T_a^2 - \beta^2 T_b^2 T_1 D)(T_{a-1}-\beta T_bT_{a+1}D)  - 2\alpha^2 \beta T_a^2 T_b T_{a+1} D.
\qe
\end{corollary}

\begin{proof}Set $Q(U)=\alpha(U+a)$ and $P(U)=\beta(U+b)$ in (\ref{eq:prop3140}).  A calculation then verifies that (\ref{eq:prop314}) and (\ref{eq:prop3140}) are equivalent operators in this case (up to a factor $\alpha$). 
\end{proof}

\begin{remark}\label{remarksec2}The proof of Proposition \ref{propsec2} involves applying certain equations to test functions of the form $Lf$, where $L$ is a linear differential operator.  This allowed us to cancel  terms to obtain  (\ref{eq:prop3140}).  We consider this technique to be a useful addition to the toolkit for finding Stein operators.  Indeed, this approach was recently used by Gaunt \cite{gaunt spl} to find Stein operators for the $H_3(Z)$ and $H_4(Z)$, where $H_n$ is the $n$-th Hermite polynomial and $Z\sim N(0,1)$.  In Section \ref{appa}, we also use the technique to derive a Stein operator for the product of independent non-centered normals with different means.  
\end{remark}

\begin{remark}We attempted to generalise Proposition \ref{propsec2} so that $X$ and $Y$ are no longer identically distributed, for which $X$ and $Y$ have Stein operators of the form $A_X=M-Q_X(MD)-P_X(MD)D$ and $A_Y=M-Q_Y(MD)-P_Y(MD)D$.  We were only able to find a Stein operator for the product $XY$ under the very restrictive condition that $P_Y(U)Q_X(U)Q_X(U+I)=P_X(U)Q_Y(U)Q_Y(U+I)$. This Stein operator  had the unusual feature of not being symmetric in $X$ and $Y$. In certain simple cases, we
  can, however, apply the proof technique of Proposition
  \ref{propsec2} to derive a Stein operator for the product of two
  non-identically distributed random variables; see Section
  \ref{appa}.  
\end{remark}


\begin{remark}Note that, whilst the Stein operator for $X$ and $Y$ in Proposition \ref{propsec2} satisfies the condition $\# \left\{ j-i\, | \, a_{ij}\neq0 \right\}= 3$, the Stein operator (\ref{eq:prop3140}) for their product satisfies $\# \left\{ j-i\, | \, a_{ij}\neq0 \right\}= 4$.  Thus, it is not possible to iterate Proposition \ref{propsec2} to find a Stein operator for product of three i.i.d$.$ random variables.  This is in contrast to the work of Gaunt, Mijoule and Swan \cite{gms18} which was carried out under the assumption $\# \left\{ j-i\, | \, a_{ij}\neq0 \right\}= 2$.
\end{remark}

\subsection{Examples}\label{sec:examples}

\subsubsection{Product of non-centered
  normals}\label{sec:product-non-centered}
  
  Assume $X$ and $Y$ are independent standard normal random variables.  A Stein operator for $X+\mu$ (or $Y+\mu$) is $A=D-M+\mu I$.  Applying Corollary \ref{prop:314}  with $\alpha=\mu$, $\beta=1$ and $a=b=\infty$ gives the following Stein operator for $Z=(X+\mu)(Y+\mu)$:
\begin{align}A_{Z}&=(M - \mu^2 I -  T_1 D)(I-D)  - 2\mu^2 D \nonumber \\
\label{equal1}&=MD^3 + (I-M) D^2 - (M+(1+\mu^2)I) D + M - \mu^2I. 
\end{align}  
(Here, and for the rest of this paper, we consider the unit variance case; the extension to general case follows from a straightforward rescaling and the resulting Stein operator for the product is given in Table \ref{table:nonlin}.) Note that when $\mu=0$ the operator becomes
\begin{align*}A_{Z}f(x)&=M(D^3-D^2-D+I)f(x)+(D^2-D)f(x) \\
&=x(f^{(3)}(x)-f''(x))+(f''(x)-f'(x))+x(f'(x)-f(x)).
\end{align*}
Taking $g(x)=f'(x)-f(x)$ then yields $A_{Z}f(x)= \tilde {A} _{Z }g (x) = xg''(x)+g'(x)-xg(x)$,
which we recognise as the product normal Stein operator that was obtained by Gaunt \cite{gaunt pn}.

\subsubsection{Product of non-centered gammas}

Assume $X$ and $Y$ are distributed as a $\Gamma(r,1)$, with p.d.f$.$ $p(x)=\frac{1}{\Gamma(r)}x^{r-1}\mathrm{e}^{-x}$, $x>0$, and let $\mu \in \R$. A Stein operator for $X+\mu$ (or $Y+\mu$) is $A = T_{r+\mu} - \mu D - M.$ Corollary \ref{prop:314} applied with $\alpha = 1$, $\beta = -\mu$, $a = r+\mu$, $b =  \infty$ yields the following fourth-order Stein operator for $Z=(X+\mu)(Y+\mu)$:
$$A_Z =  (M - T_{r+\mu}^2 - \mu^2 T_1 D)(T_{r+\mu-1}+\mu T_{r+\mu+1}D)  +2 \mu T_{r+\mu}^2  T_{r+\mu+1} D.$$
Note also that when $\mu=0$, this operator reduces to $(M-T_r^2)T_{r-1}$, which is the product gamma Stein operator of Gaunt \cite{gaunt ngb} applied to $T_{r-1}f$ instead of $f$.

\subsubsection{Product of variance-gamma random variables}\label{sec:vg}

The variance-gamma distribution with parameters $r > 0$, $\theta \in \R$, $\sigma >0$, $\mu \in \R$ has p.d.f$.$
\begin{equation}\label{vgdef}f(x) = \frac{1}{\sigma\sqrt{\pi} \Gamma(\frac{r}{2})} \mathrm{e}^{\frac{\theta}{\sigma^2} (x-\mu)} \bigg(\frac{|x-\mu|}{2\sqrt{\theta^2 +  \sigma^2}}\bigg)^{\frac{r-1}{2}} K_{\frac{r-1}{2}}\bigg(\frac{\sqrt{\theta^2 + \sigma^2}}{\sigma^2} |x-\mu| \bigg),   
\end{equation}
$x\in\R$, where $K_\nu(x)=\int_0^\infty \mathrm{e}^{-x\cosh(t)}\cosh(\nu t)\,\mathrm{d}t$, $x>0$, is the modified Bessel function of the second kind.  If a random variable $W$ has density (\ref{vgdef}) then we write $W\sim \mathrm{VG}(r,\theta,\sigma,\mu)$. A $\mathrm{VG}(r,\theta,\sigma,0)$ Stein operator is given by $\sigma^2 T_r D + 2 \theta T_{r/2} - M$ (see Gaunt \cite{gaunt vg}). Applying Corollary \ref{prop:314} with $\alpha = 2\theta, \beta = \sigma^2, a = r/2, b =  r$, we get the following Stein operator for the product of two independent $\mathrm{VG}(r,\theta,\sigma,0)$ random variables:

$$A= (M - 4\theta^2T_{r/2}^2 - \sigma^4 T_r^2 T_1 D)(T_{r/2-1}-\sigma^4 T_rT_{r/2+1}D)  - 8\theta^2 \sigma^2 T_{r/2}^2 T_r T_{r/2+1} D.$$

Note that when $\theta=0$ we have
$$Af(x)= (M  - \sigma^4 T_r^2 T_1 D)(T_{r/2-1}-\sigma^4 T_rT_{r/2+1}D)f(x).$$
Defining $g:\R\rightarrow\R$ by $xg(x)=-(T_{r/2-1}-\sigma^4 T_rT_{r/2+1}D)f(x)$ gives
\begin{align*}Ag(x)&= (\sigma^4 T_r^2 T_1 D-M)Mg(x) =\sigma^4 T_r^2T_1^2g(x)-M^2g(x),
\end{align*}
which is in agreement with the product variance-gamma Stein operator given in Section 3.2 of Gaunt, Mijoule and Swan \cite{gms18}.  Lastly, we note that the $\mathrm{VG}(r,\theta,\sigma,\mu)$ Stein operator of Gaunt \cite{gaunt vg}, as given by 
\[\sigma^2(M-\mu)D^2+(r\sigma^2+2\theta(M-\mu))D+(r\theta-(M-\mu))I,\]
 satisfies $\# \left\{ j-i\, | \, a_{ij}\neq0 \right\}= 4$ when $\mu\not=0$, and therefore one cannot apply Proposition \ref{propsec2} or Corollary \ref{prop:314} to find a Stein operator for the product of two such  variates.

\subsubsection{Product of non-identically distributed non-central normals}\label{appa}

By working on a case-by-case basis it is possible to use the proof technique of Proposition \ref{propsec2} to find Stein operators for the product of two non-identically distributed random variables, whose Stein operators satisfy the assumptions of the proposition. We find that a Stein operator for the product of independent normals $N(\mu_X,1)$ and $N(\mu_Y,1)$ is
\begin{equation}\label{unequal1}MD^4+D^3-(2M+\mu_X\mu_YI)D^2-(1+\mu_X^2+\mu_Y^2)D+M-\mu_X\mu_YI.
\end{equation}

Let us now provide a derivation of this Stein operator. Let $X$ and $Y$ be independent standard normal random variables and define $Z=(X+\mu_X)(Y+\mu_Y)$.  We will use repeatedly the fact that $\E[Wg(W)] = \E[g'(W)]$ for $W\sim N(0,1)$, as well as conditioning arguments, and we let $\E_W[]$ stand for the expectation conditioned on $W$.  Let $f:\R\rightarrow\R$ be four times differentiable and such that $\E|Zf^{(i)}(Z)|<\infty$ for $i=0,1,\ldots4$ and $\E|f^{(i)}(Z)|<\infty$ for $i=0,1,2,3$, where $f^{(0)}\equiv f$.  Then 
\begin{align}
&\E[Zf(Z)]= \E[(X+\mu_X)(Y+\mu_Y)f((X+\mu_X)(Y+\mu_Y))]\nonumber\\
&= \E[(Y+\mu_Y)\E_Y[Xf((X+\mu_X)(Y+\mu_Y)]] + \mu_X\E[(Y+\mu_Y)f(Z)]\nonumber\\
&= \E[(Y+\mu_Y)^2 f'(Z)] + \mu_X\E[(Y+\mu_Y)f(Z)]\nonumber\\
&=\E[Y(Y+\mu_Y) f'(Z)] + \mu_X\E[(Y+\mu_Y)f(Z)]+\mu_Y\E(Y+\mu_Y)f'(Z)]\nonumber\\
&=\E[f'(Z)] + \E[(X+\mu_X)(Y+\mu_Y)f''(Z)]+ \mu_X\E[(Y+\mu_Y)f(Z)]\nonumber\\
&\quad+\mu_Y\E[(Y+\mu_Y)f'(Z)]\nonumber\\
&=(1+\mu_Y^2)\E [f'(Z)]+\E[Zf''(Z)]+\mu_X\mu_Y\E[f(Z)]+\mu_X \E[Yf(Z)]\nonumber\\
\label{eq:5}&\quad+\mu_Y\E[Yf'(Z)] .
\end{align}
By again applying a conditioning argument we obtain
\begin{align*}\E[Yf(Z)]&=\E[(X+\mu_X)f'(Z)]=\mu_X\E[f'(Z)]+\E[Xf'(Z)]\\
&=\mu_X\E[f'(Z)]+\E[(Y+\mu_Y)f''(Z)]
\end{align*}
(and the same applies to $\E[Yf'(Z)]$).
Hence
\begin{align}&\E[Zf(Z)]\nonumber \\
 &=(1+\mu_Y^2)\E [f'(Z)]+\E[Zf''(Z)]+\mu_X\mu_Y\E[f(Z)]+\mu_X^2\E[f'(Z)]\nonumber \\
&\quad+\mu_X\mu_Y\E[f''(Z)]+\mu_X\E[Yf''(Z)]+\mu_Y\mu_X\E[f''(Z)]+\mu_Y^2\E[f^{(3)}(Z)]\nonumber \\
&\quad+\mu_Y\E[Yf^{(3)}(Z)]+\mu_Y\E[Yf^{(3)}(Z)]\nonumber\\
&=(1+\mu_X^2+\mu_Y^2)\E [f'(Z)]+\E[Zf''(Z)]+\mu_X\mu_Y\E[f(Z)]\nonumber\\
\label{eq:6}&\quad+2\mu_X\mu_Y\E[f''(Z)]+\mu_Y^2\E[f^{(3)}(Z)]+\mu_X\E[Yf''(Z)]+\mu_Y\E[Yf^{(3)}(Z)].
\end{align}
Isolating the expressions depending on $Y$ from   
\eqref{eq:5} and (\ref{eq:6}), we obtain two different equations:
\begin{align}\label{eq:333}\mu_X \E[Yf''(Z)]+\mu_Y\E[Yf^{(3)}(Z)] &=\E[(Z-\mu_X\mu_Y)f(Z)\nonumber\\
&\quad-(1+\mu_X^2+\mu_Y^2)f'(Z) \nonumber\\
&\quad-(Z+2\mu_X\mu_Y)f''(Z)-\mu_Y^2f^{(3)}(Z)]
\end{align}
and
\begin{align}
\mu_X\E[Yf''(Z)]+\mu_Y\E[Yf^{(3)}(Z)]&=\E[(Z-\mu_X\mu_Y)f''(Z)-(1+\mu_Y^2)f^{(3)}(Z)\nonumber\\
\label{eq:444}&\quad-Zf^{(4)}(Z)].
\end{align}
Substract \eqref{eq:444} to \eqref{eq:333} to get
\begin{align*}\E[Zf^{(4)}(Z)+f^{(3)}(Z)-(2Z+\mu_X\mu_Y)f''(Z)-(1&+\mu_X^2+\mu_Y^2)f'(Z)\\
&+(Z-\mu_X\mu_Y)f(Z)]=0,
\end{align*}
from which we deduce that (\ref{unequal1}) is a Stein operator for $Z$. \hfill $\Box$

\vspace{3mm}

Lastly, we note that applying the  operator (\ref{equal1}) to $f(x)=g'(x)+g(x)$ yields
\[xg^{(4)}(x)+ g^{(3)}(x)-(2x+\mu^2)g''(x)-(1+2\mu^2)g'(x)+(x-\mu^2)g(x),\]
which we recognise as the Stein operator (\ref{unequal1}) in the special case $\mu_X=\mu_Y=\mu$.

\subsubsection{Sums of products of normals}

Let us begin by noting a simple result, that has perhaps surprisingly not previously been stated explicitly in the literature.  Suppose $X,X_1,\ldots,X_n$ are i.i.d., with Stein operator $A_Xf(x)=\sum_{k=0}^m(a_kx+b_k)f^{(k)}(x)$, where $m\geq1$ and the $a_k$ and $b_k$ are real-valued constants.  Let $W=\sum_{j=1}^nX_j.$  Then, by conditioning,
\begin{align*}&\E[(a_0W+nb_0)f(W)]\\
&=\sum_{j=1}^n\E\bigg[\E\bigg[(a_0X_j+b_0)f(W)\,\bigg|\,X_1,\ldots,X_{j-1},X_{j+1},\ldots,X_n\bigg]\bigg] \\
&=-\sum_{j=1}^n\E\bigg[\E\bigg[\sum_{k=1}^m(a_kX_j+b_k)f^{(k)}(W)\,\bigg|\,X_1,\ldots,X_{j-1},X_{j+1},\ldots,X_n\bigg]\bigg] \\
&=-\E\bigg[\sum_{k=1}^m(a_kW+nb_k)f^{(k)}(W)\bigg].
\end{align*}
Thus, a Stein operator for $W$ is given by
\begin{equation}\label{sumop}A_Wf(x)=\sum_{k=0}^m(a_kx+nb_k)f^{(k)}(x).
\end{equation}

\begin{remark}
  Identity \eqref{sumop} actually generalises similar observations for
  score functions and Stein kernels, for which such an
  additive stability is well-known, see Nourdin, Peccati and Swan \cite{nourdinpecswan2014}.
\end{remark}

Since the coefficients in the Stein operators (\ref{equal1}) and (\ref{unequal1}) are linear, we can use (\ref{sumop}) to write down a Stein operator for the sum $W=\sum_{i=1}^rX_iY_i$, where $(X_i)_{1\leq i\leq r}\sim N(\mu_X,1)$ and $(Y_i)_{1\leq i\leq r}\sim N(\mu_Y,1)$ are independent.  When $\mu_X=\mu_Y=\mu$, we have
\begin{equation}\label{nearop}A_{W}=MD^3 + (rI-M) D^2 - (M+r(1+\mu^2)I) D + M - r\mu^2I,
\end{equation}
and when $\mu_X$ and $\mu_Y$ are not necessarily equal, we have
\begin{equation}\label{nearop2}A_{W}=MD^4+D^3-(2M+r\mu_X\mu_YI)D^2-r(1+\mu_X^2+\mu_Y^2)D+M-r\mu_X\mu_YI.
\end{equation}
When $\mu_X=\mu_Y=0$, the random variable $W$ follows the $\mathrm{VG}(r,0,1,0)$ distribution (see Gaunt \cite{gaunt vg}, Proposition 1.3).  Taking $g=f'-f$ in \eqref{nearop}  gives $A_Wg(x)=xg''(x)+rg'(x)-xg(x)$, which we recognise as the $\mathrm{VG}(r,0,1,0)$ Stein operator that was obtained in Gaunt \cite{gaunt vg}.

\begin{section}{Some results concerning the Stein operators for products of independent normal random variables}\label{sec3}
\end{section}

\subsection{On the minimality of the operators}\label{sec3.1}

The operator (\ref{eq:prop314}) is at most a seventh
order differential operator.  However, for particular cases, such as
the product of two i.i.d$.$ non-centered normals, the operator reduces
to one of lower order, see Section~\ref{sec:product-non-centered}.
Whilst we believe that the third order operator (\ref{equal1}) is a
minimal order polynomial operator, we have no proof of this claim (nor
do we have much intuition as to whether the seventh order operator
(\ref{eq:prop314}) is of minimal order). We believe this question of
minimality to be important:
\vspace{-1mm}
  \begin{con}
    \label{sec:general-result} 
There exists no second order Stein operator (acting on smooth functions with compact support) with polynomial
coefficients  for the product  of two independent non-centered normal
random variables. 
  \end{con}
\vspace{-1mm}
One can use a brute force approach to verify the conjecture for polynomials of fixed order (if the conjecture holds).  Such results would be worthwhile in practice, because a third order Stein operator with linear coefficients may be easier to use in applications than one of second order with polynomial coefficients of degree greater than one. 

Let us now us the brute force approach to prove that there is no second order Stein operator with linear coefficients for the product  of two independent non-centered normals (generalisations are obvious).       
Let $X$ and $Y$ be independent $N(1,1)$ random variables and let $Z=XY$.  Suppose that there was such a Stein operator for $Z$, then it would be of the form $A_Zf(x)=\sum_{j=0}^2(a_{0,j}+a_{1,j}x)f^{(j)}(x)$, where $f^{(0)}\equiv f$. Now, if $A_Z$ was a Stein operator for $Z$, we would have $\E[A_Zf(Z)]=0$ for all $f$ in some class $\mathcal{F}$ that contains the monomials $\{x^k\,:\,k\geq1\}$.  Taking $f(x)=x^k$, $k=0,1,\ldots,5$, we obtain six equations for six unknowns.  Letting $\mu_k$ denote $\E Z^k$, we have $\mu_1=1$, $\mu_2=4$, $\mu_3=16$, $\mu_4=100$, $\mu_5=676$ and $\mu_6=5776$.  This leads to the system of equations
\begin{align*}a_{1,0}+a_{0,0}&=0 \\
a_{1,1}+a_{0,1}+4a_{1,0}+a_{0,0}&=0\\
2a_{1,2}+2a_{0,2}+8a_{1,1}+2a_{0,1}+16a_{1,0}+4a_{0,0}&=0 \\
24a_{1,2}+6a_{0,2}+48a_{1,1}+12a_{0,1}+100a_{1,0}+16a_{0,0}&=0\\
192a_{1,2}+48a_{0,2}+400a_{1,1}+48a_{0,1}+676a_{1,0}+100a_{0,0}&=0\\
2000a_{1,2}+320a_{0,2}+3380a_{1,1}+500a_{0,1}+5776a_{1,0}+676a_{0,0}&=0.
\end{align*}
We used \emph{Mathematica} to compute that the determinant of the matrix corresponding to this system of equations is $783\,360\not=0$.  Therefore, there is a unique solution, which is clearly $a_{1,2}=\cdots=a_{0,0}=0$.  Thus, there does not exist a second order Stein operator with linear coefficients for $Z$.

Similarly, one can show that there is no third order Stein operator with linear coefficients for the product of two independent normals with different means.  Here we took $X\sim N(1,1)$ and $Y\sim N(2,1)$, and sought a Stein operator of the form $A_Zf(x)=\sum_{j=0}^3(a_{0,j}+a_{1,j}x)f^{(j)}(x)$.  We then used the monomials $f(x)=x^k$, $k=0,1,\ldots,7$, to generate eight linear equations in eight unknowns, and found the determinate of the matrix corresponding to this system of equations to be $10\,157\,222\,707\,200\not=0$. 

\subsection{Characterisation by the operators}\label{sec3.2}

We begin with a simple general result, which perhaps surprisingly has not previously been given in the literature.  The proof technique has, however, appeared in the literature; see the proof of Lemma 5.2 of Ross \cite{ross} for case of the exponential distribution.

\begin{proposition}\label{conmom}Suppose that the law of the random variable $X$, supported on $I\subset\R$, is determined by its moments.  Let the operator $A_X=\sum_{i=1}^n\sum_{j=1}^p a_{i,j}M^jD^i$, where $a_{i,j}\in\R$, act on a class of functions $\mathcal{F}$ which contains all polynomial functions.  Suppose $A_X$ is a Stein operator for $X$: that is, for all $f\in\mathcal{F}$,
\begin{equation}\label{xmom}\E[A_Xf(X)]=0.
\end{equation}
Now, let $m=\max_{i,j}(j-i)-\min_{i,j}(j-i)-1$, where the maxima and minima are taken over all $i,j$ such that $a_{i,j}\not=0$.  Suppose that the first $m$ moments of $Y$ are equal to those of $X$ and that 
\begin{equation}\label{monoeq}\E[A_Xf(Y)]=0
\end{equation}
for all $f\in\mathcal{F}$.  Then $Y$ has the same law as $X$. 
\end{proposition}

\begin{proof}We prove that all moments of $Y$ are equal to those of $X$.  As the moments of $X$ determine its law, verifying this proves the Proposition.   The monomials $\{x^k\,:\,k\geq1\}$ are contained in the class $\mathcal{F}$, so applying $f(x)=x^k$, $k\geq m$, to (\ref{monoeq}) yields the recurrence
\begin{equation}\label{receq}\sum_{i,j}a_{i,j}C_k\E Y^{k+j-i}=0, \quad k\geq m,
\end{equation}   
where $C_k=k(k-1)\cdots (k-i+1)$ if $k-i+1>0$ and $C_k=0$ otherwise.  We have that $\E Y^0=1$ and we are given that $\E Y^k=\E X^k$ for $k=1,\ldots, m$.  We can then use forward substitution in (\ref{receq}) to (uniquely) obtain all moments of $Y$.  Due to (\ref{xmom}), $\E[A_Xf(X)]=0$ for all $f\in\mathcal{F}$, and so it follows by the above reasoning that
\begin{equation*}\label{receq1}\sum_{i,j}a_{i,j}C_k\E X^{k+j-i}=0, \quad k\geq m.
\end{equation*}
But this is same recurrence relation as (\ref{receq}) and, since $\E Y^k=\E X^k$ for $k=1,\ldots, m$, it follows that $\E Y^k=\E X^k$ for all $k\geq m$ as well.
\end{proof}

If we have obtained a Stein operator $A_X$ for a random variable $X$, then Proposition \ref{conmom} tells us that the operator characterises the law of $X$ if $X$ is determined by its moments.  This characterisation is weaker than those typically found in Stein's method literature, as it involves moment conditions on the random variable $Y$.  This is perhaps not surprising, because the characterisations given in the literature have mostly been found on a case-by-case basis, whereas ours applies to a wide class of distributions.

The distribution of the product of two independent normal distributions is determined by its moments, which can be seen from the existence of its moment generating function $M(s)$ for all $|s|<1$; see Section \ref{sec3.3.1}.  The following full characterisation of the distribution is thus immediate from Proposition \ref{conmom}.

\begin{proposition}(i) Let $W$ be a real-valued random variable whose first three moments are equal to that of the random variable $Z=XY$, where $X\sim N(\mu_X,1)$ and $Y\sim N(\mu_Y,1)$ are independent.  Then $W$ is equal in law to $Z$ if and only if  
\begin{align}\label{char67} &\E \big[Wf^{(4)}(W)+ f^{(3)}(W)-(2W+\mu_X\mu_Y)f''(W)\nonumber\\
&\quad-(1+\mu_X^2+\mu_Y^2)f'(W)+(W-\mu_X\mu_Y)f(W)\big]=0
\end{align}
for all functions $f\in C^4(\R)$ such that $\E|Zf^{(j)}(Z)|<\infty$ for $0\leq j\leq 4$, and $\E|f^{(k)}(Z)|<\infty$ for $0\leq k\leq 3$, where $f^{(0)}\equiv f$.

(ii) Now suppose that $\mu_X=\mu_Y=\mu$, and that the first two moments of $W$ are equal to those of $Z$.  Then $W$ is equal in law to $Z$ if and only if  
\begin{equation*} \E\big[Wf^{(3)}(W) + (1-W) f''(W) - (W+1+\mu^2)f'(W) + (W- \mu^2)f(W)\big]=0
\end{equation*}
for all $f\in C^3(\R)$ such that $\E|Zf^{(j)}(Z)|<\infty$, $0\leq j \leq 3$, and $\E|f^{(k)}(Z)|<\infty$, $0\leq k\leq 2$.
\end{proposition}

Proposition \ref{conmom} can be used to prove that some other Stein operators given in the literature fully characterise the distribution.  For example, the Stein operator for the product of $n$ independent Beta random variables of Gaunt \cite{gaunt ngb} is characterising, since this product is supported on $(0,1)$ and thus the distribution is determined by its moments.




\subsection{Applications of the operators}\label{sec3.3}

\subsubsection{Characteristic function}\label{sec3.3.1}

As the Stein operator (\ref{unequal1}) has linear coefficients, it turns out to be straightforward to use the characterising equation (\ref{char67}) to find a formula for the characteristic function of the random variable $Z=XY$, where $X\sim N(\mu_X,1)$ and $Y\sim N(\mu_Y,1)$ are independent.  

On taking $f(x)=\mathrm{e}^{\mathrm{i}tx}$ in the characterising equation (\ref{char67}) and setting $\phi(t)=\E[\mathrm{e}^{\mathrm{i}tZ}]$, we deduce that $\phi(t)$ satisfies the differential equation
\begin{equation}\label{charmei}(t^4+2t^2+1)\phi'(t)+(-\mathrm{i}t^3+\mu_X\mu_Yt^2-(1+\mu_X^2+\mu_Y^2)\mathrm{i}t-\mu_X\mu_Y)\phi(t)=0.
\end{equation}
It should be noted that $f(x)=\mathrm{e}^{\mathrm{i}tx}$ is a complex-valued function; here we have applied the characterising equation to the real and imaginary parts of $f$, which are themselves real-valued functions. Solving (\ref{charmei}) subject to the condition that $\phi(0)=1$ then gives that
\begin{equation}\label{phioo}\phi(t)=\frac{1}{\sqrt{1+t^2}}\exp\bigg(\frac{-t(\mu_X^2t+\mu_Y^2t-2\mathrm{i}\mu_X\mu_Y)}{2(1+t^2)}\bigg).
\end{equation}
Setting $s=\mathrm{i}t$ yields a formula for the moment generating function $M(s)=\E[\mathrm{e}^{sZ}]$, which is well-defined for $|s|<1$.
We doubt these formulas are new, but it is interesting to note that we were able to obtain such a simple proof via the Stein characterisation.




\subsubsection{Probability density function}

Let $X\sim N(\mu_X,1)$ and $Y\sim N(\mu_Y,1)$ be independent, and let $Z=XY$.  For $\mu_X=\mu_Y=0$, it is a well-known and easy to prove result that the p.d.f$.$ is given by $p_Z(x)=\frac{1}{\pi}K_0(|x|)$, $x\in\R$.  However, in general, the p.d.f$.$ takes a much more complicated form (see Cui et al$.$ \cite{cui})): for $x\in\R$,
\begin{align}\label{pdfcom}p_Z(x)=\frac{1}{\pi}\mathrm{e}^{-(\mu_X^2+\mu_Y^2)/2}\sum_{n=0}^\infty\sum_{m=0}^{2n}\frac{x^{2n-m}|x|^{m-n}}{(2n)!}\binom{2n}{m}\mu_X^m \mu_Y^{2n-m}K_{m-n}(|x|), 
\end{align}

It is possible to use the Stein operators for the product $Z$ to gain insight into why there is such a dramatic increase in complexity from the zero mean case to non-zero mean case.  To see this, we recall a duality result given in Remark 2.7 of Gaunt, Mijoule and Swan \cite{gms18} (see also Section 4 of that paper for further details).  If $V$ admits a smooth density $p$, which solves the differential equation $B p = 0$ with $B = \sum_{i,j} b_{ij} M^j D^i$, then a Stein operator for $V$ is given by $A = \sum_{i,j} (-1)^i b_{ij} D^i M^j$, and similarly given a Stein operator for $V$ one can write down a differential equation satisfied by $p$. In this manner, we can write down differential equations satisfied by the density $p_W$ of the random variable $W=\sum_{i=1}^r X_iY_i$, where the $X_i$ and $Y_i$ are independent copies of $X$ and $Y$ respectively, using the Stein operators (\ref{nearop}) and (\ref{nearop2}) for this distribution.  When $\mu_X=\mu_Y=\mu$, we have
\begin{align}xp_W^{(3)}(x)+(x+3-r) p_W''(x)&-(x+r(1+\mu^2)-2)p_W'(x)\nonumber\\
\label{pode1}&\quad\quad-(x+1-r\mu^2)p_W(x) =0,
\end{align}
and in general
\begin{align}&xp_W^{(4)}(x)+3p_W^{(3)}(x)-(2x+r\mu_X\mu_Y)p_W''(x)\nonumber \\
\label{pode2}&\quad\quad+(r(1+\mu_X^2+\mu_Y^2)-4)p_W'(x)+(x-r\mu_X\mu_Y)p_W(x) =0.
\end{align}
In the special case $\mu_X=\mu_Y=0$, the density of $Z$ satisfies the modified Bessel differential equation $xp_Z''(x)+p_Z'(x)-xp_Z(x)=0$.

From Section \ref{sec3.1} and the duality result of Gaunt, Mijoule and Swan \cite{gms18}, we
know there do not exist differential equations for $p_Z$ with
linear coefficients with a lower degree than (\ref{pode1}) and
(\ref{pode2}).  Moreover, we were unable to transform (\ref{pode1}) or
(\ref{pode2}) into a well-understood class, such as the Meijer $G$-function differential
equation.  Therefore, the increase
in complexity in the p.d.f$.$ $p_Z$ of $Z$ from the zero mean to
non-zero mean case can be understood from the increase in complexity
of the differential equation satisfied by $p_Z$.  Also, due to the
above reasoning, it seems plausible that  formula (\ref{pdfcom})
cannot be simplified further.

Finally, we note that there is not a severe increase in complexity in the differential equations satisfied by $W$ from the $r=1$ case to the general case.  To the best of our knowledge, a formula for general $r\geq1$ has not been obtained in the literature, and even if the differential equations (\ref{pode1}) and (\ref{pode2}) are not ultimately used to derive such a formula, they do indicate that the formula should be at a similar level of complexity to that of (\ref{pdfcom}), and thus provide motivation for obtaining such a formula.  We note that such a result would be of interest due to the occurrence of such random variables in, for example, electrical engineering applications, see Ware and Lad \cite{ware}.

\section*{Acknowledgements} RG acknowledges support from EPSRC grant
EP/K032402/1 and is currently supported by a Dame Kathleen Ollerenshaw
Research Fellowship.  RG is grateful to Universit\'e de Li\`ege, FNRS
and EPSRC for funding a visit to University de Li\`ege, where some of
the details of this project were worked out.  YS
acknowledges support by the Fonds de la Recherche Scientifique - FNRS
under Grant MIS F.4539.16. Part of GM's research was supported by a
Welcome Grant from Universit\'e de Li\`ege. We would like to thank the referee for their careful reading of our paper and their helpful comments.

\bibliographystyle{apalike}

\end{document}